\documentclass[reqno,a4paper,draft]{amsart}

\usepackage{enumitem}
\setenumerate{label=\textnormal{(\arabic*)}}

\usepackage{amsmath,amssymb,dsfont,verbatim,bm,array,mathtools}
\usepackage[latin1]{inputenc}
\usepackage{booktabs}
\usepackage[raggedright]{titlesec}
\usepackage{mathtools}

\titleformat{\chapter}[display]
{\normalfont\huge\bfseries}{\chaptertitlename\\thechapter}{20pt}{\Huge}
\titleformat{\section}
{\normalfont\Large\bfseries\center}{\thesection}{1em}{}
\titleformat{\subsection}
{\normalfont\large\bfseries}{\thesubsection}{1em}{}
\titleformat{\subsubsection}[runin]
{\normalfont\normalsize\bfseries}{\thesubsubsection}{1em}{}
\titleformat{\paragraph}[runin]
{\normalfont\normalsize\bfseries}{\theparagraph}{1em}{}
\titleformat{\subparagraph}[runin]
{\normalfont\normalsize\bfseries}{\thesubparagraph}{1em}{}
\titlespacing*{\chapter} {0pt}{50pt}{40pt}
\titlespacing*{\section} {0pt}{3.5ex plus 1ex minus .2ex}{2.3ex plus .2ex}
\titlespacing*{\subsection} {0pt}{3.25ex plus 1ex minus .2ex}{1.5ex plus .2ex}
\titlespacing*{\subsubsection}{0pt}{3.25ex plus 1ex minus .2ex}{1.5ex plus .2ex}
\titlespacing*{\paragraph} {0pt}{3.25ex plus 1ex minus .2ex}{1em}
\titlespacing*{\subparagraph} {\parindent}{3.25ex plus 1ex minus .2ex}{1em}

\input xypic
\xyoption{all}

\newtheorem{theorem}{Theorem}[section]

\newtheorem{proposition}[theorem]{Proposition}
\newtheorem{corollary}[theorem]{Corollary}

\theoremstyle{definition}
\newtheorem{definition}[theorem]{Definition}

\theoremstyle{remark}
\newtheorem{remark}[theorem]{Remark}

\newtheorem{question}[theorem]{Question}
\newtheorem{answer}[theorem]{Answer}

\DeclareMathOperator{\Jac}{Jac}
\DeclareMathOperator{\Ima}{Im}

\begin{document}
\title{There are no Keller maps having prime degree field extensions}

\author{Vered Moskowicz}
\address{Department of Mathematics, Bar-Ilan University, Ramat-Gan 52900, Israel.}
\email{vered.moskowicz@gmail.com}

\begin{abstract}
The two-dimensional Jacobian Conjecture says that a Keller map $f: (x,y) \mapsto (p,q) \in k[x,y]^2$, 
$\Jac(p,q) \in k^*$, is an automorphism of $k[x,y]$.
We prove that there is no Keller map with $[k(x,y): k(p,q)]$ prime.
\end{abstract}

\maketitle

\section{Introduction}

The two-dimensional Jacobian Conjecture, raised by O. H. Keller \cite{keller}, says that a Keller map 
$f: (x,y) \mapsto (p,q) \in k[x,y]^2$, 
$\Jac(p,q) \in k^*$, is an automorphism of $k[x,y]$, namely, $k[p,q]=k[x,y]$.

Nice sources on the Jacobian Conjecture are \cite{essen book}, \cite{essen believe} and \cite{bcw}.

We prove that there is no Keller map with $[k(x,y): k(p,q)]$ prime.
Our proof is divided into two cases $xy \notin k[p,q]$, Theorem \ref{my1}, 
and $xy \in k[p,q]$, Theorem \ref{my2};
in each case we show that the assumption that $[k(x,y):k(p,q)]$ is prime implies that
$f$ is an automorphism, $k[p,q]=k[x,y]$, hence $k(p,q)=k(x,y)$, so $[k(x,y):k(p,q)]=1$, 
and there is no such map. 

In our proof we rely on several known results and on an answer to one of our questions in MO \cite{mo}.

\section{Known results}

In this section we recall known results that are used in our proof;
without one or more of those results it seems that we would not have been able to prove our result.
Also, we present our MO question.

Throughout this note, $k=\mathbb{C}$ and $f: (x,y) \mapsto (p,q) \in k[x,y]^2$ is a Keller map, namely, 
$\Jac(p,q)=p_xq_y-p_yq_x \in k^*$.

\subsection{Formanek's results} 

\cite[Theorem 1]{formanek two notes}:
\begin{theorem}\label{formanek1}
If $k[p,q][w]=k[x,y]$, for some $w \in k[x,y]$, then $f$ is an automorphism,
namely, $k[p,q]=k[x,y]$.
\end{theorem}

\begin{corollary}\label{cor formanek1}
If $k[p,q][x]=k[x,y]$, then $k[p,q]=k[x,y]$
and 
if $k[p,q][y]=k[x,y]$, then $k[p,q]=k[x,y]$.
\end{corollary}

\begin{proof}
Take $w=x$ or $w=y$.
\end{proof}

The following result is already written in the form needed for us;
we do not need its more general form for more then two variables.

\cite[Theorem 2]{formanek observations}:
\begin{theorem}\label{formanek2}
$k(p,q,x)=k(x,y)$ and $k(p,q,y)=k(x,y)$.
\end{theorem}

\begin{remark}\label{remark formanek}
Actually, if $w=g(x)$ for some automorphism $g$ of $k[x,y]$,
then $k(p,q,w)=k(x,y)$, see the discussion in \cite[page 370]{formanek observations}.
For example, $k(p,q,x+y)=k(x,y)$.
\end{remark}

\subsection{Wang's result}

\cite[Theorem 41(i)]{wang}:
\begin{theorem}\label{flat}
$k(p,q) \cap k[x,y] = k[p,q]$.
\end{theorem}

\subsection{Results of Jedrzejewicz and Zieli\'{n}ski}    

\begin{definition}\label{def1}
Let $A$ be an integral domain, $R \subseteq A$ a subring.
$R \subseteq A$ is called \textit{square-factorially closed} in $A$ 
if the following condition is satisfied:
For $u \in A$ arbitrary and $v \in A$ square-free, if $u^2v \in R-\{0\}$, then $u,v \in R$.
\end{definition}

The following result inspired us to discover Theorem \ref{my1}.
\cite[Theorem 3.4]{nice}:
\begin{theorem}\label{root1}
Let $A$ be a UFD, $R \subseteq A$ a subring of $A$
such that $R^*=A^*$ (invertible elements) and $F(R) \cap A=R$,
where $F(R)$ is the field of fractions of $R$.

TFAE:
\begin{itemize}
\item The set of square-free elements of $R$ is contained in the set of square-free elements of $A$.
\item $R$ is square-factorially closed in $A$.
\end{itemize}
\end{theorem}

Then we have:
\begin{theorem}\label{square-factorially closed}
$k[p,q]$ is square-factorially closed in $k[x,y]$.
\end{theorem}

\begin{proof}
We can apply Theorem \ref{root1}, since $k[x,y]$ is a UFD,
$k[p,q]^*=k[x,y]^*$ and $k(p,q) \cap k[x,y]=k[p,q]$ (Theorem \ref{flat})
and the set of square-free elements of $k[p,q]$ 
is contained in the set of square-free elements of $k[x,y]$
by \cite[Theorem 2.4]{nice}, $(i)$ implies $(iii)$.
\end{proof}

Theorem \ref{square-factorially closed} says: 
For $u \in k[x,y]$ arbitrary and $v \in k[x,y]$ square-free, 
if $u^2v \in k[p,q]-\{0\}$, then $u,v \in k[p,q]$; 
we will apply this property several times in the proof of Theorem \ref{my1}.

\begin{definition}\label{def2}
Let $A$ be an integral domain, $R \subseteq A$ a subring.
$R \subseteq A$ is called \textit{root closed} in $A$ 
if the following condition is satisfied:
For every $u \in A$ and $n \geq 1$, if $u^n \in R$, then $u \in R$.
\end{definition}

\cite[Theorem 3.6]{nice}:
\begin{theorem}\label{root2}
Let $A$ be a UFD, $R \subseteq A$ a subring of $A$
such that $R^*=A^*$ (invertible elements) and $F(R) \cap A=R$,
where $F(R)$ is the field of fractions of $R$.

If $R$ is square-factorially closed in $A$, then $R$ is root closed in $A$.
\end{theorem}

Then we have:
\begin{theorem}\label{root closed}
$k[p,q]$ is root closed in $k[x,y]$.
\end{theorem}

\begin{proof}
We can apply Theorem \ref{root2}, since $k[x,y]$ is a UFD,
$k[p,q]^*=k[x,y]^*$ and $k(p,q) \cap k[x,y]=k[p,q]$ (Theorem \ref{flat})
and $k[p,q]$ is square-factorially closed in $k[x,y]$ 
by Theorem \ref{square-factorially closed}.
\end{proof}

Theorem \ref{root closed} says: 
For $u \in k[x,y]$ and $n \geq 1$, if $u^n \in k[p,q]$, then $u \in k[p,q]$;
we will apply this property several times in the proof of Theorem \ref{my1}.

\subsection{Galois extension}

\cite[Theorem 2.1]{bcw}, with $(g)$ implies $(a)$:
\begin{theorem}\label{galois}
If $k(p,q) \subseteq k(x,y)$ is Galois, then $f$ is an automorphism.
\end{theorem}

Then,
\begin{corollary}\label{galois cor}
If $[k(x,y):k(p,q)]=2$, then $f$ is an automorphism.
\end{corollary}

\begin{proof}
It is well-known that an extension of degree two is Galois, hence by Theorem \ref{galois},
$f$ is an automorphism.
\end{proof}

\subsection{Keller's theorem}

\begin{theorem}[Keller's theorem]\label{keller thm}
If $k(x,y)=k(p,q)$, then $k[x,y]=k[p,q]$, namely, $f$ is an automorphism.
\end{theorem}

\subsection{Injectivity on one line} 

Two results concerning injectivity; we will apply both results in Theorem \ref{my2}.

\cite[page 284]{essen embedding}:
\begin{definition}\label{def embedding}
A polynomial map $k \ni t \mapsto g(t)=(g_1(t),\ldots,g_n(t)) \in k^n$
is called an \textit{embedding} of $k$ in $k^n$ if via $g$ $k$ is isomorphic to its image
i.e. there exists a polynomial map $G:k^n \to k$ such that $g$ and $G \big|_{\Ima g}$
are each others inverses.

In algebraic terms we get: 
$g$ is an embedding if and only if $k[g_1(t),\ldots,g_n(t)]=k[t]$.
\end{definition}

\begin{proposition}\label{embedding}
$g$ is an embedding if and only if $g'(t) \neq \bar{0}$ for all $t \in k$ 
and the map $g: k \to k^n$ is injective.
\end{proposition}

\cite[Theorem 1.1]{inj on one line}:
\begin{theorem}\label{inj}
Let $H: k^2 \to k^2$ be a polynomial mapping such that $\Jac(H) \in k^*$.
If there exists a line $l \subset k^2$ such that $H \big|_l : l \to k^2$ is injective
then $H$ is a polynomial automorphism.
\end{theorem}

\subsection{Common zeros of two polynomials} 

\cite[Theorem 11.9.10]{artin}:
\begin{theorem}\label{common zeros}
Let $f=f(x,y), g=g(x,y) \in k[x,y]$ be two nonzero polynomials in two variables.
Then $f$ and $g$ have only finitely many common zeros in $k^2$,
unless they have a common nonconstant factor in $k[x,y]$.
\end{theorem}

\begin{proposition}\label{common zeros jc}
$p$ and $q$ have only finitely many common zeros in $k^2$.
\end{proposition}

\begin{proof}
By Theorem \ref{common zeros} $p$ and $q$ have only finitely many common zeros in $k^2$,
unless they have a common nonconstant factor in $k[x,y]$.

Therefore, we wish to show that $p$ and $q$ do not have a common nonconstant factor in $k[x,y]$.

Otherwise, $r=r(x,y) \in k[x,y]$ is a common nonconstant factor of $p$ and $q$,
namely, $d_{1,1}(r) \geq 1$, $p=r\tilde{p}$ and $q=r\tilde{q}$, 
where $\tilde{p}, \tilde{q} \in k[x,y]$.

On the one hand, $\Jac(p,q) \in k^*$.

On the other hand, 
$\Jac(p,q) = \Jac(r\tilde{p},r\tilde{q}) = \ldots = 
r[r\Jac(\tilde{p},\tilde{q})+\tilde{q}\Jac(\tilde{p},r)+\tilde{p}\Jac(r,\tilde{q})]= rw$,
where $w=r\Jac(\tilde{p},\tilde{q})+\tilde{q}\Jac(\tilde{p},r)+\tilde{p}\Jac(r,\tilde{q}) \in k[x,y]$.

Then, $k^* \ni \Jac(p,q)= rw$ with $r,w \in k[x,y]$ and $d_{1,1}(r) \geq 1$, which is impossible
($rw \in k^*$ implies that $r,w \in k^*$, which contradicts $d_{1,1}(r) \geq 1$).

Our assumption that $r=r(x,y) \in k[x,y]$ is a common nonconstant factor of $p$ and $q$
yields an impossible situation, hence $p$ and $q$ do not have a common nonconstant factor in $k[x,y]$.
\end{proof}

\subsection{An extension with 'many' primitive elements}
Finally, we present our question \cite{mo}.

\begin{question}\label{mo question}
Let $R \subseteq k(x,y)$ and assume that $R=k(u,v)$, 
where $u,v \in k[x,y]$ are algebraically independent over $k$.

Here $\mathbb{N}$ includes $0$.

Assume that $R$ satisfies the following 'rare' property:
For every monomial $x^iy^j$, $i \in \mathbb{N}$, $j \in \mathbb{N}$ 
(except the case $i=j=0$, for which we assume nothing), 
we have $k(u,v,x^iy^j)=k(x,y)$.

\textbf{Question:} Is it true that $R=k(x,y)$?

I am not able to find a counterexample, but perhaps there is such.

I do not mind to further assume that $x+y$ also satisfies $k(u,v,x+y)=k(x,y)$.

Any help is welcome! Thank you very much.
\end{question}

In the question, $u$ and $v$ are algebraically independent over $k$,
hence $\Jac(u,v) \in k[x,y]-\{0\}$; 
there is no assumption that they have an invertible Jacobian, 
so any nonzero polynomial as a Jacobian is fine.

\begin{answer}\label{mo answer}
Without considering the additional condition $k(u,v,x+y)=k(x,y)$, 
it was proved in the answer that: $[k(x,y):R]=[k(x,y):k(u,v)]=2$.
\end{answer}


\section{Our result}



\begin{definition}['Rare property']\label{rare property}
Here $\mathbb{N}$ includes $0$.
For every $i,j \in \mathbb{N}$, denote by $C_{i,j}$ the following property:
$k(p,q,x^iy^j)=k(x,y)$.

If for every $(i,j) \in \mathbb{N} \times \mathbb{N} - \{(0,0)\}$,
$C_{i,j}$ holds, then we say that $k[p,q]$ satisfies \textit{the rare property}.
\end{definition}

We are ready to prove the first case, which says:
``There is no Keller map $(x,y) \mapsto (p,q)$ having prime degree field extension 
and $xy \notin k(p,q)$".

\begin{theorem}[First Case]\label{my1}
Assume that:
\begin{itemize}
\item $[k(x,y):k(p,q)]=P$, for some prime number $P$.
\item $xy \notin k(p,q)$.
\end{itemize}
Then $f$ is an automorphism, $k[p,q]=k[x,y]$, hence $k(p,q)=k(x,y)$, so $[k(x,y):k(p,q)]=1$, 
and there is no such map. 
\end{theorem}

\begin{proof}
We will show that for every $(i,j) \in \mathbb{N} \times \mathbb{N} - \{(0,0)\}$, $C_{i,j}$ holds,
namely, $k(p,q,x^iy^j)=k(x,y)$; in other words, we will show that $k[p,q]$ satisfies the rare property,
definition \ref{rare property}.

Having this it is immediate that $f$ is an automorphism; indeed,
Answer \ref{mo answer} implies that $[k(x,y):k(p,q)]=2$, so by Corollary \ref{galois cor}, 
$f$ is an automorphism.

We will show now that for every $(i,j) \in \mathbb{N} \times \mathbb{N} - \{(0,0)\}$, $C_{i,j}$ holds,
dividing $\mathbb{N} \times \mathbb{N} - \{(0,0)\}$ into several subsets:

\textbf{Case 1:} $(i,j) \in \{(1,0),(0,1)\}$: 

$C_{1,0}$ and $C_{0,1}$ say that $k(p,q,x)=k(x,y)$ and $k(p,q,y)=k(x,y)$, respectively,
and these results are true for any Keller map by Theorem \ref{formanek2}.

\textbf{Case 2:} $(i,j) \in \{(n,0)\}_{n \geq 2} \cup \{(0,n)\}_{n \geq 2}$: 

Fix $n \geq 2$. 

If $x^n \in k(p,q)$, then $x^n \in k(p,q) \cap k[x,y] = k[p,q]$ (Theorem \ref{flat}),
so by Theorem \ref{root closed}, $x \in k[p,q]$.
Then $x \in k(p,q)$, which implies that $k(p,q)=k(p,q,x)=k(x,y)$, by Theorem \ref{formanek2}.
We obtained $k(p,q)=k(x,y)$, hence $[k(x,y):k(p,q)]=1$ contrary to our assumption that 
$[k(x,y):k(p,q)]$ is prime.

Therefore, $x^n \notin k(p,q)$, and since $[k(x,y):k(p,q)]$ is prime, we get that
$k(p,q,x^n)=k(x,y)$, since every element $w \notin k(p,q)$ of a prime degree extension must be
a primitive element for that extension.

$k(p,q,x^n)=k(x,y)$ is condition $C_{n,0}$ and we are done.

Similarly for $C_{0,n}$.

\textbf{Case 3:} $(i,j)$, $ij \neq 0$, namely, each of $\{i,j\}$ is non-zero.

$ij \neq 0$ means that $i \geq 1$ and $j \geq 1$.

Divide into four sub-cases; in each case we will obtain that $k(p,q,x^iy^j)=k(x,y)$, 
namely, $C_{i,j}$ holds, for the $i,j$'s dealt in that specific sub-case.

\textbf{(1) Sub-case ee:} Both $i$ and $j$ are even. 

If $i=j$, then $x^iy^i \notin k(p,q)$, since otherwise, $(xy)^i=x^iy^i \in k(p,q)$,
then $(xy)^i \in k(p,q) \cap k[x,y]=k[p,q]$ (Theorem \ref{flat}),
so by root closedness, Theorem \ref{root closed}, $xy \in k[p,q] \subset k(p,q)$, 
contrary to our assumption that $xy \notin k(p,q)$.

Therefore, $x^iy^i \notin k(p,q)$, so $x^iy^i$ is a primitive element,
$k(p,q,x^iy^i)=k(x,y)$, since $[k(x,y):k(p,q)]=P$ is prime,
and we obtained $C_{i,i}$.

Next, if $i \neq j$, write $i=2^nN$ and $j=2^mM$, where each of $N,M$ is odd.
In other words, $n$ is the highest power of $2$ in $i$ and $m$ is the highest power of $2$ is $m$. 
$N$ and $M$ may not be co-prime, but we do not need them to be co-prime, just odd numbers.
It may happen that $n=m$ (but not $n=m$ and $N=M$ simultaneously, which would imply that $i=j$).
W.l.o.g. $n \geq m$

$x^iy^j=x^{2^nN}y^{2^mM}=x^{2^{n-m+m}N}y^{2^mM}$
$=x^{2^{n-m}2^mN}y^{2^mM}=(x^{2^{n-m}N})^{2^m}(y^M)^{2^m}$
$=(x^{2^{n-m}N}y^M)^{2^m}$,
concluding that $x^iy^j=(x^{2^{n-m}N}y^M)^{2^m}$


We will show that $x^iy^j \notin k(p,q)$. 

Otherwise, $(x^{2^{n-m}N}y^M)^{2^m}=x^iy^j \in k(p,q)$.
Then $(x^{2^{n-m}N}y^M)^{2^m} \in k(p,q) \cap k[x,y] = k[p,q]$, 
hence Theorem \ref{root2} implies that 
$x^{2^{n-m}N}y^M \in k[p,q]$.

If $n-m=0$, then $x^Ny^M \in k[p,q]$, with $N$ and $M$ odd.

Then, 
$k[p,q] \ni x^Ny^M=x^{N-1+1}y^{M-1+1}=x^{N-1}xy^{M-1}y=(x^{N-1}y^{M-1})(xy)$.
Each of $N-1$ and $M-1$ is even, including zero (it may happen that one of $\{N-1,M-1\}$ is zero or both are zero),
so write $N-1=2s$ and $M-1=2t$, for some $s,t \in \mathbb{N}$.

Then, 
$k[p,q] \ni x^Ny^M=(x^{N-1}y^{M-1})(xy)$
$=(x^{2s}y^{2t})(xy)=(x^sy^t)^2(xy)$.
Now apply Theroem \ref{root1} with $u=x^sy^t$ and $v=xy$ 
($xy$ is indeed square-free in $k[x,y]$)
to conclude that $x^sy^t \in k[p,q]$ and $xy \in k[p,q]$.

We obtained $xy \in k[p,q] \subset k(p,q)$, but we assumed that $xy \notin k(p,q)$, 
therefore $x^iy^j \notin k(p,q)$.
Then, $k(p,q,x^iy^j)=k(x,y)$, because $[k(x,y):k(p,q)]=P$, $P$ prime.

If $n-m>0$, then $d:=n-m > 0$, 
hence $x^{2^{n-m}N}y^M=x^{2^dN}y^M$,
with $2^dN$ even and $M$ odd (and $N$ odd).
 
$x^{2^dN}y^M=x^{2^{d-1+1}N}y^M$
$=x^{2^{d-1}2^1 N}y^M$
$=x^{2 2^{d-1} N}y^{M-1+1}$
$=(x^{2^{d-1}N})^2 y^{M-1}y$.

$M-1$ is even, so write $M-1=2t$, for some $t \in \mathbb{N}$.
We continue with our computation, 
$(x^{2^{d-1}N})^2 y^{M-1}y=(x^{2^{d-1}N})^2 y^{2t}y$
$=(x^{2^{d-1}N})^2 (y^t)^2 y=(x^{2^{d-1}N} y^t)^2 y$.

Summarizing, $k[p,q] \ni x^{2^{n-m}N}y^M = (x^{2^{d-1}N} y^t)^2 y$.

Now apply Theroem \ref{root1} with $u=x^{2^{d-1}N} y^t$ and $v=y$ ($y$ is indeed square-free in $k[x,y]$)
to conclude that $x^{2^{d-1}N} y^t \in k[p,q]$ and $y \in k[p,q]$.

But $y \in k[p,q]$ implies that $k[p,q][x]=k[x,y]$, 
hence Corollary \ref{cor formanek1} says that $k[p,q]=k[x,y]$,
so $k(p,q)=k(x,y)$, contradicting $[k(x,y):k(p,q)]=P$, $P$ prime.

Therefore, $x^iy^j \notin k(p,q)$.
Then, $k(p,q,x^iy^j)=k(x,y)$, because $[k(x,y):k(p,q)]=P$, $P$ prime.


\textbf{(2) Sub-case oo:} Both $i$ and $j$ are odd.
Then each of $\{i-1,j-1\}$ is even, so we can write $i-1=2s$ and $j-1=2t$,
for some $s,t \in \mathbb{N}$ (it may happen that one of $\{s,t\}$ is zero or both).

$x^iy^j \notin k(p,q)$, since otherwise, $x^iy^j \in k(p,q) \cap k[x,y] = k[p,q]$.

$k[p,q] \ni x^iy^j=x^{i-1+1}y^{j-1+1}=x^{i-1}xy^{j-1}y$
$=x^{2s}xy^{2t}y=(x^s)^2x(y^t)^2y=(x^sy^t)^2(xy)$.

Summarizing, $k[p,q] \ni x^iy^j=(x^sy^t)^2(xy)$.

Now apply Theroem \ref{root1} with $u=x^sy^t$ and $v=xy$ ($xy$ is indeed square-free in $k[x,y]$)
to conclude that $x^sy^t \in k[p,q]$ and $xy \in k[p,q]$.

We obtained $xy \in k[p,q] \subset k(p,q)$, but we assumed that $xy \notin k(p,q)$, 
therefore $x^iy^j \notin k(p,q)$.
Then, $k(p,q,x^iy^j)=k(x,y)$, because $[k(x,y):k(p,q)]=P$, $P$ prime.

\textbf{(3) Sub-case eo:} $i$ is even and $j$ is odd:
$i$ is even, so we can write $i=2s$, for some $s \in \mathbb{N}$ ($s$ may equal zero).
$j$ is odd, then we can write $j=2t+1$, for some $t \in \mathbb{N}$ ($t$ may equal zero).

$x^iy^j \notin k(p,q)$, since otherwise, $x^iy^j \in k(p,q) \cap k[x,y] = k[p,q]$.

$k[p,q] \ni x^iy^j=x^{2s}y^{2t+1}=x^{2s}y^{2t}y$
$=(x^s)^2(y^t)^2y=(x^sy^t)^2y$.

Summarizing, $k[p,q] \ni x^iy^j=(x^sy^t)^2y$.

Now apply Theroem \ref{root1} with $u=x^sy^t$ and $v=y$ ($y$ is indeed square-free in $k[x,y]$)
to conclude that $x^sy^t \in k[p,q]$ and $y \in k[p,q]$.

But $y \in k[p,q]$ implies that $k[p,q][x]=k[x,y]$, hence Corollay ref{cor formanek1} says that $k[p,q]=k[x,y]$,
so $k(p,q)=k(x,y)$, contradicting $[k(x,y):k(p,q)]=P$, $P$ prime.

Therefore, $x^iy^j \notin k(p,q)$.
Then, $k(p,q,x^iy^j)=k(x,y)$, because $[k(x,y):k(p,q)]=P$.

\textbf{(4) Sub-case oe:} $i$ is odd and $j$ is even: Similar to the third sub-case.

\end{proof}

We prove now the second case, which says:
``There is no Keller map $(x,y) \mapsto (p,q)$ having prime degree field extension 
and $xy \in k(p,q)$".

\begin{theorem}[Second Case]\label{my2}
Assume that:
\begin{itemize}
\item $[k(x,y):k(p,q)]=P$, for some prime number $P$.
\item $xy \in k(p,q)$.
\end{itemize}
Then $f$ is an automorphism, $k[p,q]=k[x,y]$, hence $k(p,q)=k(x,y)$, so $[k(x,y):k(p,q)]=1$, 
and there is no such map.
\end{theorem}

The proof is different from the proof of Theorem \ref{my1}.

\begin{proof}
By assumption, $xy \in k(p,q)$, so $xy \in k(p,q) \cap k[x,y] = k[p,q]$, by Theorem \ref{flat}.

Therefore, $xy=H(p,q)$, for some $H=H(T_1,T_2) \in k[T_1,T_2]$, 
where $k[T_1,T_2]$ is a polynomial ring in two variables $T_1,T_2$ over $k$.

By Proposition \ref{common zeros jc}, $p$ and $q$ have only finitely many common zeros, 
list them $C=\{(\lambda_0,\mu_0),\ldots,(\lambda_L,\mu_L)\}$, 
$\lambda_i, \mu_i \in k$, $0 \leq i \leq L$.

Take $\mu \notin \{\mu_0,\ldots,\mu_L\}$; this means that for every $\lambda \in k$, 
$(\lambda,\mu)$ is not a common zero of $p$ and $q$.

Write $p=p(x,y)=p_n(y)x^n+p_{n-1}(y)x^{n-1}+\cdots+p_1(y)x+p_0(y)$, 
where $p_i(y) \in k[y]$, $0 \leq i \leq n$, $p_n \neq 0$,
and write $q=q(x,y)=q_m(y)x^m+q_{m-1}(y)x^{m-1}+\cdots+q_1(y)x+q_0(y)$, 
where $q_i(y) \in k[y]$, $0 \leq i \leq m$, $q_m \neq 0$.

\textbf{Step 1:} Each of $\{p(x,\mu), q(x,\mu)\}$ is not identically zero.

Otherwise, if $p(x,\mu) \equiv 0$, then 
$p_n(\mu)x^n+p_{n-1}(\mu)x^{n-1}+\cdots+p_1(\mu)x+p_0(\mu) \equiv 0$, 

hence all the coefficients, $p_{i}(\mu)$, $0 \leq i \leq n$, are zero: 
$p_n(\mu)=p_{n-1}(\mu)=\cdots=p_1(\mu)=p_0(\mu)=0$.

For every $0 \leq i \leq n$, $p_i(y) \in k[y]$ has $\mu \in k$ as a root: $p_i(\mu)=0$,
so for every $0 \leq i \leq n$, $p_i=(y-\mu)a_i$, for some $a_i \in k[y]$.

Then $p$ becomes:
$p=(y-\mu)a_nx^n+(y-\mu)a_{n-1}x^{n-1}+\cdots+(y-\mu)a_1x+(y-\mu)a_0$
$=(y-\mu)(a_nx^n+a_{n-1}x^{n-1}+\cdots+a_1x+a_0)$.

Denote $r=a_nx^n+a_{n-1}x^{n-1}+\cdots+a_1x+a_0$,
so $p=(y-\mu)r$.

A direct computation of the Jacobian of $p=(y-\mu)r$ and $q$,
shows that $\Jac(p,q)=(y-\mu)\Jac(r,q)-rq_x$.

But $\Jac(p,q)=c \in k^*$, so $c=(y-\mu)\Jac(r,q)-rq_x$.

Take $y=\mu$ on both sides and get: $c=-r(x,\mu)q_x(x,\mu)$,
hence $q_x(x,\mu) \equiv e$, where $e$ is a nonzero scalar.

{}From $q=q(x,y)=q_m(y)x^m+q_{m-1}(y)x^{m-1}+\cdots+q_1(y)x+q_0(y)$,
we get that $q_x=mq_m(y)x^{m-1}+(m-1)q_{m-1}(y)x^{m-2}+\cdots+q_1(y)$,
with $mq_m(y) \neq 0$.

Then $q_x(x,\mu) \equiv e$ says that 
$mq_m(\mu)x^{m-1}+(m-1)q_{m-1}(\mu)x^{m-2}+\cdots+q_1(\mu) \equiv e$,

so 
$mq_m(\mu)x^{m-1}+(m-1)q_{m-1}(\mu)x^{m-2}+\cdots+q_1(\mu)-e \equiv 0$,
hence all the coefficient are zero:

$mq_m(\mu)=(m-1)q_{m-1}(\mu)=\ldots=2q_2(\mu)=(q_1(\mu)-e)=0$.

Then, $q_m(\mu)=q_{m-1}(\mu)=\ldots=q_2(\mu)=0$ and $q_1(\mu)=e$.

Therefore, $q_m(y)=(y-\mu)b_m, q_{m-1}(y)=(y-\mu)b_{m-1},\ldots,q_2(y)=(y-\mu)b_2$,

for some $b_m=b_m(y),b_{m-1}=b_{m-1}(y),\ldots,b_2=b_2(y) \in k[y]$.

For $q_1(\mu)-e=0$ we get $(q_1-e)(\mu)=0$, so $q_1-e=(y-\mu)b_1$,

for some $b_1=b_1(y) \in k[y]$, hence $q_1=(y-\mu)b_1+e$.

Then $q$ becomes 
$q=(y-\mu)b_mx^m+(y-\mu)b_{m-1}x^{m-1}+\cdots+(y-\mu)b_2x^2+((y-\mu)b_1+e)x+q_0$,
so
$q=(y-\mu)b_mx^m+(y-\mu)b_{m-1}x^{m-1}+\cdots+(y-\mu)b_2x^2+(y-\mu)b_1x+ex+q_0$
$=(y-\mu)(b_mx^m+b_{m-1}x^{m-1}+\cdots+b_2x^2+b_1x)+ex+q_0$.

Denote $s=b_mx^m+b_{m-1}x^{m-1}+\cdots+b_2x^2+b_1x$,
hence, $q=(y-\mu)s+ex+q_0$.

Summarizing, we have $p=(y-\mu)r$ and $q=(y-\mu)s+ex+q_0$, 
where $e \in k^*$, $q_0 \in k[y]$, $r,s \in k[x,y]$, 
and $\mu \in k$ has the property that for every $\lambda \in k$, 
$(\lambda,\mu)$ is not a common zero of $p$ and $q$.

If we take $y=\mu$ in $p$ and $q$ we obtain:
$p(x,\mu)=(\mu-\mu)r(x,\mu)=0$ 
and 
$q(x,\mu)=(\mu-\mu)s(x,\mu)+ex+q_0(\mu)=ex+q_0(\mu)$.

The polynomial $ex+q_0(\mu) \in k[x]$ has $x$-degree one,
so it has a root $\tilde{\lambda} \in k$:
$\tilde{\lambda}=-\frac{q_0(\mu)}{e} \in k$.

Therefore, $(\tilde{\lambda},\mu)$ is a common zero of $p$ and $q$:
$p(\tilde{\lambda},\mu)=(\mu-\mu)r(\tilde{\lambda},\mu)=0$ 
and 
$q(\tilde{\lambda},\mu)=(\mu-\mu)s(\tilde{\lambda},\mu)+e(\tilde{\lambda})+q_0(\mu)=0$.

This contradicts our choice of $\mu$ having the property that for every $\lambda \in k$, 
$(\lambda,\mu)$ is not a common zero of $p$ and $q$.


\textbf{Step 2:} Each of $\{p(x,\mu), q(x,\mu)\}$ is not identically a nonzero constant.

Otherwise, if $p(x,\mu) \equiv c$, for some $c \in k^*$,
then $p(x,\mu)-c \equiv 0$, hence we consider the Jacobian pair $(p-c,q)$
instead of the Jacobian pair $(p,q)$ and same arguments as in step 1
show that $p(x,\mu) \equiv c$ is impossible.

Having step 1 and step 2 we conclude: $\deg_x(p(x,\mu)) \geq 1$ and $\deg_x(q(x,\mu)) \geq 1$.

In $xy=H(p,q)$ substitute on both sides $y$ by $\mu$, 
hence $\mu x=H(p(x,\mu),q(x,\mu))$, 
which shows that $k[x]=k[p(x,\mu),q(x,\mu)]$.

(The efforts in step 1 and step 2 were to make sure there are no problematic cases
where $p(x,\mu),q(x,\mu) \in k$).

Define $g: k \to k^2$, $g: x \mapsto g(x)=(p(x,\mu),q(x,\mu))$.
(Here $g_1(x)=p(x,\mu)$ and $g_2(x)=q(x,\mu)$).

We have just seen that $k[x]=k[p(x,\mu),q(x,\mu)]=k[g_1(x),g_2(x)]$,
so by definition \ref{def embedding} $g$ is an embedding.

Then Proposition \ref{embedding} implies that $g$ is injective.

It is clear that our given 
$f: k^2 \to k^2$, $f: (x,y) \mapsto (p(x,y),q(x,y))$
is injective on the line $l$, $l: {y=\mu}$.

Indeed, $f \big|_l: l \to k^2$, 
$f \big|_l (x,\mu) \mapsto (p(x,\mu),q(x,\mu))$
is exactly $g$, which is injective.

By Theorem \ref{inj} $f$ is an automorphism.

\end{proof}

We hope that our ideas presented in this note will contribute to solving the two-dimensional 
Jacobian Conjecture.

\section{Acknowledgements}
I would like to thank Laurent Moret-Bailly for answering \cite{mo}.
Also, I would like to thank MSE and MO users that answered my questions over the years.


\bibliographystyle{plain}

\end{document}